\documentclass[a4paper]{article}

\usepackage[top=0.8in, bottom=0.8in, left=1in, right=1in]{geometry}
\usepackage[utf8]{inputenc}
\usepackage[T1]{fontenc}
\usepackage[frenchb,english]{babel}
\usepackage{dsfont}
\usepackage{amssymb,amsfonts,amsthm,amsmath}
\usepackage{graphicx}
\usepackage{enumerate,enumitem}
\usepackage{hyperref}
\hypersetup{
  pdftitle={cFK-local-limit},
  pdfauthor={Chen-Linxiao},
  colorlinks=true,
  citecolor=blue,
  linkcolor=blue
}
\usepackage{multirow}
\usepackage{caption,subcaption}
\usepackage{amscd}

\usepackage{mymathsymb}
\usepackage{mythm}
\usepackage{abABF}

\newcommand{\refnote}[2]{\hyperref[#1]{\ref*{#1}#2}}

\newcommand*{\bpar}[1]{\big(#1\big)}
\newcommand*{\Bpar}[1]{\Big(#1\Big)}
\newcommand*{\map}{\mathbf}
\newcommand*{\M}{\mathcal{M}}
\newcommand*{\W}{\mathcal{W}}
\newcommand*{\wt}{\widetilde}

\newcommand*{\m}{\map{M}}
\newcommand*{\QM}{Q_{\map{M}}}

\newcommand*{\QQ}{\proba[Q]\otimes Q}
\newcommand*{\ve}[1][e]{\vec{#1}}

\newcommand*{\ven}[1][]{(\vec{e}_{n#1})_{n\ge 0}}

\title{Basic properties of the infinite critical-FK random map}

\author{Linxiao Chen\thanks{IPhT, CEA-Saclay and Universit\'e Paris-Sud, E-mail: linxiao.chen@math.u-psud.fr}}
\date{}
\begin{document}
\maketitle

\abstract{%
In this paper we investigate the critical Fortuin-Kasteleyn (cFK) random map model.
For each $q\in[0,\infty]$ and integer $n\geq 1$, this model chooses a planar map of $n$ edges with a probability proportional to the partition function of critical $q$-Potts model on that map.
Sheffield introduced the hamburger-cheeseburger bijection which maps the cFK random maps to a family of random words, and remarked that one can construct infinite cFK random maps using this bijection.
We make this idea precise by a detailed proof of the local convergence. 
When $q=1$, this provides an alternative construction of the UIPQ.
In addition, we show that the limit is almost surely one-ended and recurrent for the simple random walk for any $q$, and mutually singular in distribution for different values of $q$.
}

\bigskip\noindent
\textbf{Keywords.} Fortuin-Kasteleyn percolation, random planar maps, hamburger-cheeseburger bijection, local limits, recurrent graph, ergodicity of random graphs.

\medskip\noindent
\textbf{Mathematics Subject Classification (2010).}
60D05, 
60K35, 
05C81, 
60F20. 

\section{Introduction}

\paragraph{Planar maps.}
Random planar maps has been the focus of intensive research in recent
years.  We refer to \cite{ADJ97} for the physics background and
motivations, and to \cite{StFlour14} for a survey of recent
results in the field.

A \emph{finite planar map} is a proper embedding of a finite connected graph into the two-dimensional sphere, viewed up to orientation-preserving homeomorphisms.
\emph{Self-loops} and \emph{multiple edges} are allowed in the graph.
In this paper we will not deal with non-planar maps, and thus we drop the adjective ``planar'' sometimes.
The \emph{faces} of a map are the connected components of the complement of the embedding in the sphere and the \emph{degree} of a face is the number of edges incident to it.
A map is a \emph{triangulation} (resp.~a \emph{quadrangulation}) if all of its faces are of degree three (resp.~four).
The \emph{dual map} $\m^\dual$ of a planar map $\map{M}$ has one vertex associated to each face of $\map{M}$ and there is an edge between two vertices if and only if their corresponding faces in $\map{M}$ are adjacent.

\begin{figure}[t]
\captionsetup{width=0.8\textwidth}
\begin{center}
\includegraphics[scale=1]{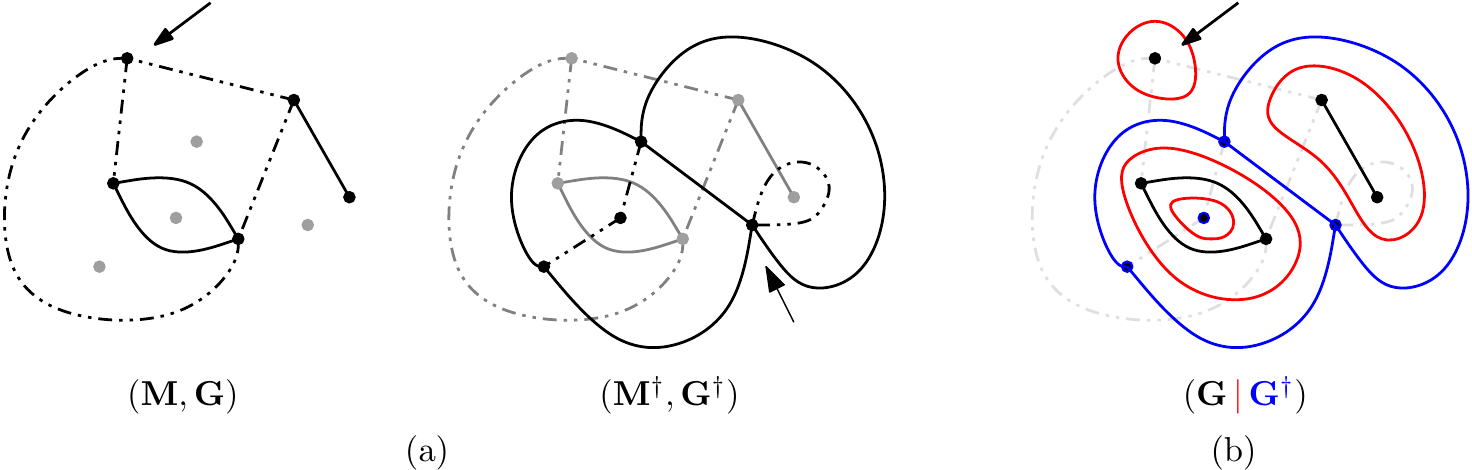}
\caption{%
(a) A subgraph-rooted map and its dual.
Edges of the distinguished subgraph are drawn in solid line, and the other edges in dashed line. The root corner is indicated by an arrow.
(b) Loops separating the distinguished subgraph and its dual subgraph.
}
\label{fig:subgraph-rooted map}
\end{center}
\end{figure}

A \emph{corner} in a planar map is the angular section delimited by two consecutive half-edges around a vertex.
It can be identified with an oriented edge using the orientation of the sphere.
A \emph{rooted} map is a map with a distinguished oriented edge or, equivalently, a corner.
We call \emph{root edge} the distinguished oriented edge, and \emph{root vertex} (resp.\ \emph{root face}) the vertex (resp.\ face) incident to the distinguished corner.
Rooting a map on a corner (instead of the more traditional choice of rooting on an oriented edge) allows a canonical choice of the root for the dual map: the dual root is obtained by exchanging the root face and the root vertex.
A \emph{subgraph} of a planar map is a graph consisting of a subset of its edges and all of its vertices.
Given a subgraph $\map{G}$ of a map $\map{M}$, the \emph{dual subgraph} of $\map{G}$, denoted by $\map{G}^\dual$, is the subgraph of $\m^\dual$ consisting of all the edges that do not intersect $\map{G}$.
Following the terminology in \cite{Ber07}, we call \emph{subgraph-rooted map} a rooted planar map with a distinguished subgraph.
Fig.~\refnote{fig:subgraph-rooted map}{(a)} gives an example of a subgraph-rooted map with its dual map.

\paragraph{Local limit.} For subgraph-rooted maps, the local distance is defined by
\begin{equation}
	d\sub{loc}(\map{(M,G),(M',G')}) = 
	\inf\set{2^{-R}}{R\in\natural,\, B_R(\map{M,G})=B_R(\map{M',G'})}
\end{equation}
where $B_R(\map{M,G})$, the ball of radius $R$ in $(\map{M,G})$, is the subgraph-rooted map consisting of all vertices of $\map{M}$ at graph distance at most $R$ from the root vertex and the edges between them.
An edge of $B_R(\map{M,G})$ belongs to the distinguished subgraph of $B_R(\map{M,G})$ if and only if it is in $\map{G}$.
The space of all finite subgraph-rooted maps is not complete with respect to $d\sub{loc}$ and we denote by $\M$ its Cauchy completion.
We call \emph{infinite subgraph-rooted map} the elements of $\M$ which are not finite subgraph-rooted map.
Note that with this definition all infinite maps are \emph{locally finite}, that is, every vertex is of finite degree.

The study of infinite random maps goes back to the works of Angel, Benjamini and Schramm on the Uniform Infinite Planar Triangulation (UIPT) \cite{BS01,AS03} obtained as the local limit of uniform triangulations of size tending to infinity. Since then variants of this theorem have been proved for different classes of maps \cite{CD06,Kri06,Men10,CMM13,BS14}. A common point of the these infinite random lattices is that they are constructed from the uniform distribution on some finite sets of planar maps.
In this work, we consider a different type of distribution.

\paragraph{cFK random map.}
For $n \geq 1$ we write $\M_n$ for the set of all subgraph-rooted maps with $n$ edges. Recall that in a dual subgraph-rooted maps, the distinguished subgraph $\map{G}$ and its dual subgraph $\map{G}^\dual$ do not intersect.
Therefore we can draw a set of loops tracing the boundary between them, as in Fig.~\refnote{fig:subgraph-rooted map}{(b)}.
Let $\ell(\map{M,G})$ be the \emph{number of loops separating $\map{G}$ and $\map{G}^\dual$}.
For each $q>0$, let $\proba[Q]\ind{q}_n$ be the probability distribution on $\M_n$ defined by
\begin{equation}
	\proba[Q]\ind{q}_n(\map{M,G})\, \propto\, q^{\frac{1}{2}\ell(\map{M,G})}
\end{equation}
By taking appropriate limits, we can define $\proba[Q]\ind{q}_n$ for $q\in\{0,\infty\}$. A \emph{critical Fortuin-Kasteleyn (cFK) random map} of size $n$ and of parameter $q$ is a random variable of law $\proba[Q]\ind{q}_n$ (see Equation~\eqref{eq:marginal} below for the connection with the Fortuin-Kasteleyn random cluster model). From the definition of the loop number $\ell$, it is easily seen that the law $\proba[Q]\ind{q}_n$ is self-dual (which is why we call it critical):
\begin{equation}\label{eq:self dual}
	\proba[Q]\ind{q}_n(\map{M,G}) = \proba[Q]\ind{q}_n(\map{M^\dual,G^\dual})
\end{equation}
Our main result is:
\begin{thm}\label{thm:map limit}
For each $q\in[0,\infty]$,  we have $\proba[Q]\ind{q}_n \cvg{}{n\to\infty} \proba[Q]\ind{q}_\infty$ in distribution with respect to the metric $d\sub{loc}$.
Moreover, if $(\map{M},\map{G})$ has law $\proba[Q]\ind{q}_\infty$, then 
\begin{itemize}
\item $( \mathbf{M}, \mathbf{G}) = ( \mathbf{M}^\dagger,  \mathbf{G}^\dagger)$ in distribution,
\item the map $\map{M}$ is almost surely one-ended and recurrent for the simple random walk,
\item the laws of $\mathbf{M}$ for different values of $q$ are mutually singular.
\end{itemize}
\end{thm}

So far two classes of methods have been developed to prove local convergence of finite random maps. The first one, initially used in \cite{AS03} is based on precise asymptotic enumeration formulas for certain classes of maps. Although enumeration results about (a generalization of) cFK decorated maps have been obtained using combinatorial techniques \cite{Kos89,EK95,BB11,GJSZ12,BBG12a,BBG12b,BBG12c}, we are not going to follow this approach here.
Instead, we will first transform our finite map model through a bijection into simpler objects. The archetype of such bijection is the famous Cori-Vauqulin-Schaeffer bijection and its generalizations \cite{Sch98,BDG04}.
Then we take local limits of these simpler objects and construct the limit of the maps directly from the latter.
This technique has been used e.g.~in \cite{CD06,CMM13,BS14}.
In this work the role of the Schaeffer bijection will be played by Sheffield's hamburger-cheeseburger bijection \cite{She11} which maps a cFK random map to a \emph{random word} in a measure-preserving way. We will then construct the local limit of cFK random maps by showing that the random word converges locally to a limit, and that the hamburger-cheeseburger bijection has an almost surely continuous extension for that limit.

The cFK random maps have also been the subject of the recent works \cite{BLR15} and \cite{GMS15,GS15a,GS15b}.
These works focused on finer properties of the infinite cFK random map such as exact scaling exponents or scaling limit of the model.
In particular, the scaling exponents associated with the length and the enclosed area of a loop in the infinite cFK random map were derived independently.
The main purpose of the present paper is to prove the local convergence of finite cFK maps to the infinite cFK map.
We offer a detailed proof and construct explicitly the infinite-volume version of the hamburger-cheeseburger bijection.
The one-endedness and recurrence of the infinite cFK random map are obtained as a by-product of this bijection.
The fact that the joint law of $(\map{M,G})$ is mutually singular for different $q$ follows from the various scaling exponents computed in \cite{BLR15} and \cite{GMS15}.
By replacing the law of $(\map{M,G})$ by its marginal in $\map{M}$, we improve slightly the result.
Our proof is based on an ergodicity result of the cFK random maps, which is of independent interest (See Appendix \ref{appendix}).

The rest of this paper is organized as follows.
In Section \ref{sec:cFK} we discuss the law of the cFK random map in more details and examine three interesting special cases. In Section \ref{sec:words} we first define the random word model underlying the hamburger-cheeseburger bijection. Then we show that the model has an explicit local limit, and we prove some properties of the limit. In Section \ref{sec:maps} we construct the hamburger-cheeseburger bijection and prove Theorem \ref{thm:map limit} by translating the properties of the infinite random word in terms of the maps.

\paragraph{Acknowledgements.}
I deeply thank Jérémie Bouttier and Nicolas Curien for their many helps in completing this work.
I am grateful to Ewain Gwynne for pointing out a mistake in the previous version of this paper.
I thank the anonymous referees for many valuable comments, especially for suggesting the mutual singularity property stated in Theorem 1.
I thank also the Isaac Newton Institute and the organizers of the Random Geometry programme for their hospitality during the completion of this work.
This work is partly supported by Grant ANR-12-JS02-0001 (ANR CARTAPLUS) and by Grant ANR-14-CE25-0014 (ANR GRAAL).

\section{More on cFK random map} \label{sec:cFK}
Let $( \mathbf{M}, \mathbf{G})$ be a subgraph-rooted map and denote by $c(\map{G})$ the number of connected components in $\map{G}$.
Recalling the definition of $\ell( \mathbf{M}, \mathbf{G})$ given in the introduction, it is not difficult to see that $\ell( \mathbf{M}, \mathbf{G})=c(\map{G})+c(\map{G}^\dual)-1$.
However $c(\map{G}^\dual)$ is nothing but the number of faces of $\map{G}$, therefore by Euler's relation we have
\begin{equation}\label{eq:loop number}
	\ell(\map{M,G}) = e(\map{G}) + 2c(\map{G}) - v(\map{M}),
\end{equation}
where $e(\map{G})$ is the number of edges $\map{G}$, and $v(\map{M})$ is the number of vertices in $\map{M}$.
This gives the following expression of the first marginal of $\proba[Q]\ind{q}_n$:
for all rooted map $\map{M}$ with $n$ edges, we have
\begin{equation}\label{eq:marginal}
	\proba[Q]\ind{q}_n(\map{M}) \propto q^{-\frac{1}{2} v(\map{M})}
		\sum_{\map{G\subset M}} \sqrt{q}^{e(\map{G})} q^{ c(\map{G})}.
\end{equation}
The sum on the right-hand side over all the subgraphs of $\map{M}$ is
precisely the partition function of the Fortuin-Kasteleyn random
cluster model or, equivalenty, of the Potts model on the map
$\map{M}$
(The two partition functions are equal. See e.g.~\cite[Section 1.4]{Gri06}.
See also \cite[Section 2.1]{BBG12c} for a review of their connection with loop models on planar lattices).
For this reason, the cFK random map is used as a model of quantum gravity theory in which the geometry of the space interacts with the matter (spins in the Potts model).
Note that the ``temperature'' in the Potts model and the prefactor $q^{-\frac{1}{2}v(\map{M})}$ in \eqref{eq:marginal} are tuned to ensure self-duality, which is crucial for our result to hold.

Three values of the parameter $q$ deserve special attention, since the cFK random map has nice combinatorial interpretations in these cases.

\vspace{-1ex}
\paragraph{$\mathbf{q=0}$:}
$\proba[Q]\ind{0}_n$ is the uniform measure on the elements of $\M_n$ which minimize the number of loops $\ell$.
The minimum is $\ell\sub{min}=1$ and it is achieved if and only if the subgraph $\map{G}$ is a spanning tree of $\map{M}$.
Therefore under $\proba[Q]\ind{0}_n$, the map $\map{M}$ is chosen with probability proportional to the number of its spanning trees, and conditionally on $\map{M}$, $\map{G}$ is a uniform spanning tree of $\map{M}$.

At the limit, the marginal law of $\map{G}$ under $\proba[Q]\ind{0}_\infty$ will be that of a critical geometric Galton-Walton tree conditioned to survive.
This will be clear once we defined the hamburger-cheeseburger bijection.
In fact when $q=0$, the hamburger-cheeseburger bijection is reduced to a bijection between tree-rooted maps and excursions of simple random walk on $\integer^2$ introduced earlier by Bernardi \cite{Ber07}.

\vspace{-1ex}
\paragraph{$\mathbf{q=1}$:}
$\proba[Q]\ind{1}_n$ is the uniform measure on $\M_n$.
Since each planar map with $n$ edges has $2^n$ subgraphs, $\map{M}$ is a uniform planar map chosen among the maps with $n$ edges.
Thus in the case $q=1$, Theorem \ref{thm:map limit} can be seen as a construction of the Uniform Infinite Planar Map or of the Uniform Infinite Planar Quadrangulation via Tutte's bijection.
It is a curious fact that with this approach, one has to first decorate a uniform planar map with a random subgraph in order to show the local convergence of the map.
As we will see later, the couple $(\map{M,G})$ is encoded by the hamburger-cheeseburger bijection in an entangled way.

\vspace{-1ex}
\paragraph{$\mathbf{q=\infty}$:}
Similarly to the case $q=0$, the probability $\proba[Q]_n\ind{\infty}$ is the uniform measure on the elements of $\M_n$ which \emph{maximize} $\ell$.
To see what are these elements, remark that each connected component of $\map{G}$ contains at least one vertex, therefore
\begin{equation}
	c(\map{G}) \leq v(\map{M})
\end{equation}
And, at least one edge must be removed from $\map{M}$ to create a new connected component, so
\begin{equation}
	c(\map{G}) \leq c(\map{M}) + e(\map{M}) - e(\map{G}) = n+1 -e(\map{G})
\end{equation}
Summing the two relations, we see that the maximal number of loops is $\ell\sub{max}=n+1$ and it is achieved if and only if each connected component of $\map{G}$ contains exactly one vertex (i.e.\ all edges of $\map{G}$ are self-loops) \emph{and} that the complementary subgraph $\map{M\backslash G}$ is a tree.
Fig.~\refnote{fig:loop tree}{(a)} gives an example of such couple $(\map{M,G})$.

\begin{figure}[ht!]
\captionsetup{width=0.8\textwidth}
\begin{center}
\includegraphics[scale=1]{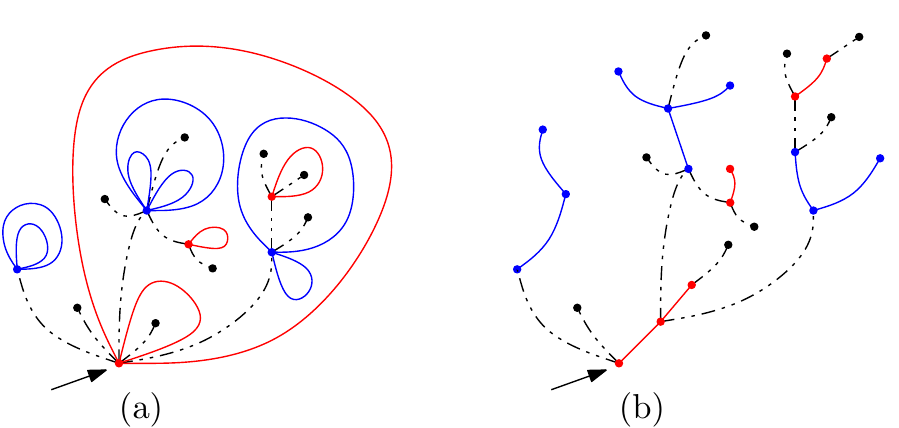}
\caption{(a) A subgraph-rooted map which maximizes the number of loops $\ell$.
Colors are used only to illutrate the bijection.
(b) The percolation configuration on a rooted tree associated to this map by the bijection.
The divided vertices, as well as the replaced edges, are drawn in the same color before and after the bijection.
}
\label{fig:loop tree}
\end{center}
\end{figure}

This model of loop-decorated tree is in bijection with bond percolation of parameter $1/2$ on a uniform random plane tree with $n$ edges, as we now explain.
For a couple $(\map{M,G})$ satisfying the above conditions,
consider a self-loop $\map{e}$ in $\map{G}$.
This self-loop separates the rest of the map $\map{M}$ into two parts which share only the vertex of $\map{e}$.
We divide this vertex in two, and replace the self-loop $\map{e}$ by an edge joining the two child vertices.
The new edge is always considered part of $\map{G}$.
By repeating this operation for all self-loops in the subgraph $\map{G}$ in an arbitrary order, we transform the map $\map{M}$ into a rooted plane tree, see Fig.~\ref{fig:loop tree}.
This gives a bijection from the support of $\proba[Q]\ind{\infty}_n$ to the set of rooted plane tree of $n$ edges with a distinguished subgraph.
The latter object converges locally to a critical geometric Galton-Watson tree conditioned to survive, in which each edge belongs to the distinguished subgraph with probability $1/2$ independently from other edges.
Using the inverse of the bijection above (which is almost surely continuous at the limit), we can explicit the law $\proba[Q]\ind{\infty}_\infty$.
In particular, it is easily seen that $\map{M}$ is almost surely a one-ended tree plus finitely many self-loops at each vertex.
Therefore it is one-ended and recurrent.

\section{Local limit of random words}\label{sec:words}
In this section we define the random word model underlying the hamburger-cheeseburger bijection, and establish its local limit.

We consider words on the alphabet $\Theta=\{\xa,\xb,\xA,\xB,\xF\}$.
Formally, a \emph{word} $w$ is a mapping from an interval $I$ of integers to $\Theta$.
We write $w\in\Theta^I$ and we call $I$ the \emph{domain} of $w$.
Let $\W$ be the space of all words, that is,
\begin{equation}
\W = \bigcup_I \Theta^I
\end{equation}
where $I$ runs over all subintervals of $\integer$.
Note that a word can be finite, semi-infinite or bi-infinite.
We denote by $\emptystr$ the \emph{empty word}.
Given a word $w$ of domain $I$ and $k\in I$, we denote by $w_k$ the letter of index $k$ in $w$.
More generally, if $J$ is an (integer or real) interval, we denote by $w_J$ the \emph{restriction} of the word $w$ to $I\cap J$.
For example, if $w = \xb\xA\xb\xa\xF\xA\xB\xa\in \Theta^{\{0,\ldots,7\}}$, then $w_{[2,6)} = \xb\xa\xF\xA \in\Theta^{\{2,3,4,5\}}$. We endow $  \mathcal{W}$ with the local distance 
\begin{equation}\label{eq:def Dloc}
	D\sub{loc}(w,w') = \inf \set{ 2^{-R} }{ R\in\natural,\, w_{[-R,R)} = w'_{[-R,R)} }
\end{equation}
Note that the equality $w_{[-R,R)} = w'_{[-R,R)}$ implies that $I\cap [-R,R) = I'\cap [-R,R)$, where $I$ (resp. $I'$) is the domain of the word $w$ (resp. $w'$). It is easily seen that $(\W,D\sub{loc})$ is a compact metric space.

\subsection{Reduction of words}
Now we define the reduction operation on the words.
For each word $w$, this operation specifies a pairing between letters in the word called $\emph{matching}$, and returns two shorter words $\rdb{w}$ and $\rdc{w}$.

We follow the exposition given in \cite{She11}.
The letters $\xa,\xb,\xA,\xB,\xF$ are interpreted as, respectively, a hamburger, a cheeseburger, a hamburger order, a cheeseburger order and a \emph{flexible} order.
They obey the following order fulfillment relation: a hamburger order $\xA$ can only be fulfilled by a hamburger $\xa$, a cheeseburger order $\xB$ by a cheeseburger $\xb$, while a flexible order $\xF$ can be fulfilled either by a hamburger $\xa$ or by a cheeseburger $\xb$.
We write $\lambda=\{\xa,\xb\}$ and $\Lambda=\{\xA,\xB,\xF\}$ for the set of lowercase letters (burgers) and uppercase letters (orders).

\paragraph{Finite case.} 
A finite word $w\in\Theta^I$ can be seen from left to right as a sequence of events that happen in a restaurant with time indexed by $I$.
Namely, at each time $k\in I$, either a burger is produced, or an order is placed.
The restaurant puts all its burgers on a stack $S$, and takes note of unfulfilled orders in a list $L$.
Both $S$ and $L$ start as the empty string.
When a burger is produced, it is appended at the end of the stack.
When an order arrives, we check if it can be fulfilled by one of the burgers in the stack.
If so, we take the \emph{last} such burger in the stack and fulfills the order. (That is, the stack is last-in-first-out.)
Otherwise, the order goes to the end of the list $L$.
Fig.~\ref{fig:arch diagram} illustrates this dynamics with an example.

\begin{figure}[h]
\begin{center}
\raisebox{1.3cm}{$
\arraycolsep=2pt
\begin{array}{l| llllll llllll llllll l}
\hline
w_k &&			\xa		&&	\xa		&&	\xB		&&	\xb		&&	\xA		&&	\xF	&&	\xB		&&	\xF		&&	\xa		&
\\\hline
S	&\emptystr	&&\xa		&&\xa\xa	&&\xa\xa	&&\xa\xa\xb	&&\xa\xb	&&\xa	&&\xa		&&\emptystr	&&\xa
\\\hline
L	&\emptystr	&&\emptystr	&&\emptystr	&&\xB		&&\xB		&&\xB		&&\xB	&&\xB\xB	&&\xB\xB	&&\xB\xB	
\\\hline
\end{array}
$}\hspace{1cm}
\includegraphics[scale=1.2]{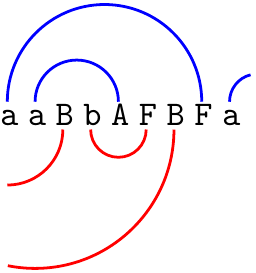}
\caption{The reduction procedure of a word and the associated arch diagram}
\label{fig:arch diagram}
\end{center}
\end{figure}

We encode the \emph{matching} of $w$ by a function $\phi_w:I\to I\cup\{-\infty,\infty\}$.
If the burger produced at time $j$ is consumed by an order placed at time $k$, then the letters $w_j$ and $w_k$ are said to be \emph{matched}, and we set $\phi_w(j)=k$ and $\phi_w(k)=j$.
On the other hand, if a letter $w_k$ corresponds to a unfulfilled order or a leftover burger, then it is \emph{unmatched}, and we set $\phi_w(k)=\infty$ if it is a burger ($w_k\in \lambda$) and $\phi_w(k)=-\infty$ if it is an order ($w_k \in \Lambda$).

Moreover, let us denote by $\rdc{w}$ (resp.\ $\rdb{w}$) the state of the list $L$ (resp. the stack $S$) at the end of the day.
Together they give the reduced form of the word $w$.

\begin{definition}[reduced word]\label{def:reduced word}
The reduced word associated to a finite word $w$ is the concatenation $\rd{w} = \rdc{w}\rdb{w}$.
That is, it is the list of unmatched uppercase letters in $w$, followed by the list of unmatched lowercase letters in $w$.
\end{definition}

The matching and the reduced word can be represented as an arch diagram as follows.
For each letter $w_j$ in the word $w$, draw a vertex in the complex plane at position $j$.
For each pair of matched letters $w_j$ and $w_k$, draw a semi-circular arch that links the corresponding pair of vertices.
This arch is drawn in the upper half plane if it is incident to an $\xa$-vertex, and in the lower half plane if it is incident to a $\xb$-vertex.
For an unmatched letter $w_j$, we draw an open arch from $j$ tending to the left if $\phi_w(j)=-\infty$, or to the right if $\phi_w(j)=\infty$.
See Fig.~\ref{fig:arch diagram}.

It should be clear from the definition of matching operation that the arches in this diagram do not intersect each other.
We shall come back to this diagram in Section \ref{sec:maps} to construct the hamburger-cheeseburger bijection.

\paragraph{Infinite case.}
Remark that a hamburger produced at time $j$ is consumed by a hamburger order at time $k>j$ if and only if 1) all the hamburgers produced during the interval $[j+1,k-1]$ are consumed strictly before time $k$, and 2) all the hamburger or flexible orders placed during $[j+1,k-1]$ are fulfilled by a burger produced strictly after time $j$.
In terms of the reduced word, this means that two letters $w_j=\xa$ and $w_k=\xA$ are matched if and only if $\rd{w_{(j,k)}}$ does not contain any $\xa$, $\xA$ or $\xF$.
This can be generalized to any pair of burger/order.

\begin{prop}[\cite{She11}]\label{prop:local}
For $j<k$, assume that $w_j\in\lambda$ and $w_k\in\Lambda$ can be matched.
Then they are matched in $w$ if and only if $\rd{w_{(j,k)}}$ does not contain any letter that can be matched to either $w_j$ or $w_k$.
\end{prop}
\noindent
This shows that the matching rule is entirely determined by the reduction operator.
More importantly, we see that the matching rule is \emph{local}, that is, whether $\phi_w(j)=k$ or not only depends on $w_{[j,k]}$.
From this we deduce that the reduction operator is compatible with string concatenation, that is, $\rd{uv} = \rd{u\rd{v}} = \rd{\rd{u}v}$ for any pair of finite words $u$, $v$.

This locality property allows us to define $\phi_w$ for infinite words $w$.
Then, we can also read $\rdb{w}$ (resp. $\rdc{w}$) from $\phi_w$ as the (possibly infinite) sequence of unmatched lowercase (resp. uppercase) letters.
However $\rd{w}$ is not defined in general, since the concatenation $\rdc{w}\rdb{w}$ does not always make sense.

\paragraph{Random word model and local limit.}
For each $p\in[0,1]$, let $\theta\ind{p}$ be the probability measure on $\Theta$ such that
\begin{align*}
	\theta\ind{p}(\xa) &= \theta\ind{p}(\xb) = \frac{1}{4}		&
	\theta\ind{p}(\xA) &= \theta\ind{p}(\xB) = \frac{1-p}{4}	&
	\theta\ind{p}(\xF) &= \frac{p}{2}
\end{align*}
Here $p$ should be interpreted as the proportion of flexible orders among all the orders.
Remark that, regardless of the value of $p$, the distribution is symmetric when exchanging $\xa$ with $\xb$ and $\xA$ with $\xB$.
As we will see in Section \ref{sec:maps}, this corresponds to the self-duality of cFK random maps.

For $n\geq 1$, let $I_k = \{-k,\ldots,2n-1-k\}$, and set
\begin{equation}
\W_n = \bigcup_{0\leq k<2n} \set{w\in \Theta^{I_k} }{\rd{w} = \emptystr}
\end{equation}
For $p\in [0,1]$, let $\proba\ind{p}_n$ be the probability measure on $\W_n$ proportional to the direct product of $\theta\ind{p}$, that is, for all $w\in\W_n$,
\begin{equation}
	\proba\ind{p}_n(w) \propto \prod_{j} \theta\ind{p}(w_j)
\end{equation}\vspace{-0.7ex}\\
where the product is taken over the domain of $w$.
In addition, let $\proba\ind{p}_\infty = \left.\theta\ind{p}\right.^{\otimes\integer}$ be the product measure on bi-infinite words. Our proof of Theorem \ref{thm:map limit} relies mainly on the following proposition, stated by Sheffield in an informal way in \cite{She11}.
\begin{prop} \label{prop:word limit}
For all $p\in[0,1]$, we have $\proba\ind{p}_n \to \proba\ind{p}_\infty$ in law for $D\sub{loc}$ as $n \to \infty$.
\end{prop}

\subsection{Proof of Proposition \ref{prop:word limit}}

We follow the approach proposed by Sheffield in \cite[Section 4.2]{She11}. Let $W\ind{p}$ be a random word of law $\proba\ind{p}_\infty$, so that $W_{[0,n)}\ind{p}$ is a word of length $n$ with i.i.d.\ letters.
By compactness of $(\W,D\sub{loc})$, it suffices to show that for any ball $B\sub{loc}$ in this space, we have $\proba\ind{p}_n(B\sub{loc}) \!\longrightarrow \!
				\proba\ind{p}_\infty (B\sub{loc})$.
Note that $D\sub{loc}$ is an ultrametric and 
the ball $B\sub{loc}(w,2^{-R})$ of radius $2^{-R}$ around $w$ is the set of words which are identical to $w$ when restricted to $[-R,R)$.	
In the rest of the proof, we fix an integer $R\geq 1$ and a word $w\in\Theta^{[-R,R)\cap \integer}$.
Recall that $W^{(p)}$ has law $ \mathbb{P}_{\infty}^{(p)}$.
In the following we omit the parameter $p$ from the superscripts to keep simple notations.

Recall that the space $\W_n$ is made up of $2n$ copies of the set $\set{w\in\Theta^{2n}}{\rd{w}=\emptystr}$ differing from each other by translation of the indices.
Therefore $\proba_n$ can be seen as the conditional law of $W_{[-K,2n-K)}$ on the event  $\{\rd{W_{[-K,2n-K)}}=\emptystr \}$, where $K$ is a uniform random variable on $\{0,\ldots,2n-1\}$ independent from $W$.
Moreover, for the word $W_{[-K,2n-K)}$ to have $w$ as its restriction to $[-R,R)$, one must to have $R\leq K\leq 2n-R$.
Hence,
\begin{align*}
		\proba_n\Par{ B\sub{loc}(w,2^{-R}) } &=	\proba\Pcond{	R\leq K\leq 2n-R\text{ and }W_{[-R,R)} = w
					}{	\rd{W_{[-K,2n-K)}} = \emptystr 	}
\\&=	\frac{1}{2n} \sum_{k=R}^{2n-R}
		\proba\Pcond{	W_{[-R,R)} = w
					}{	\rd{W_{[-k,2n-k)}} = \emptystr 	}
\\&=	\frac{1}{2n} \sum_{k=0}^{2n-2R}
		\proba\Pcond{	W_{[k,k+2R)}\simeq w }{ \rd{W_{[0,2n)}}=\emptystr }
\\&=	\expect\Econd{	\frac{1}{2n} \sum_{k=0}^{2n-2R} \idd{ W_{[k,k+2R)}\simeq w }
					}{	\rd{W_{[0,2n)}}=\emptystr }
\end{align*}
where in the last two steps, we denote by $u\simeq v$ the fact that two words are equal
up to an overall translation of indices.
On the other hand, set
\begin{equation}
	\pi_w = \proba(B\sub{loc}(w,2^{-R})) = \prod_{k=-R}^{R-1} \theta(w_k)
\end{equation}
By translation invariance of $W$ we have
\begin{equation}
\expect\Sqpar{ \frac{1}{2n} \sum_{k=0}^{2n-2R} \idd{W_{[k,k+2R)}\simeq w} }
= \frac{2n-2R+1}{2n} \pi_w
\end{equation}
In fact, up to boundary terms of the order $O(R/n)$, the quantity inside the expectation
is the empirical measure of the Markov chain $(W_{[k,k+2R)})_{k\geq 0}$ taken at the state $w$.
This is an irreducible Markov chain on the finite state space $\Theta^{2R}$.
Sanov's theorem (see e.g.~\cite[Theorem 3.1.2]{Dem98}) gives the following large deviation estimate.
For any $\epsilon>0$, there are constants $A_\epsilon, C_\epsilon>0$ depending only on $\epsilon$
and on the transition matrix of $(W_{[k,k+2R)})_{k\geq 0}$, such that
\begin{equation}
	\proba\Par{ \Big|\frac{1}{2n} \sum_{k=0}^{2n-2R} \idd{W_{[k,k+2R)} \simeq w} - \pi_w \Big|
				>\epsilon }
\leq A_\epsilon e^{-C_\epsilon n}
\end{equation}
for all $n\geq 1$.
Since $\abs{\frac{1}{2n} \sum\limits_{k=0}^{2n-2R} \idd{W_{[k,k+2R)} \simeq w} - \pi_w}$
is bounded by 1, we have
\begin{align*}
	 	\abs{ \proba_n\Par{ B\sub{loc}(w,2^{-R}) } -\pi_w }
&\leq	\expect\Econd{ 
		\Big| \frac{1}{2n} \sum_{k=0}^{2n-2R} \idd{W_{[k,k+2R)} \simeq w} - \pi_w \Big|
		}{ \rd{W_{[0,2n)}}=\emptystr }
\\&\leq \epsilon + \proba\Pcond{ 
		\Big| \frac{1}{2n} \sum_{k=0}^{2n-2R} \idd{W_{[k,k+2R)} \simeq w} - \pi_w \Big|
		> \epsilon
		}{ \rd{W_{[0,2n)}}=\emptystr }
\\&\leq	\epsilon + \frac{1}{ \proba\Par{ \rd{W_{[0,2n)}}=\emptystr }} \cdot
		\proba\Par{
		\Big| \frac{1}{2n} \sum_{k=0}^{2n-2R} \idd{W_{[k,k+2R)} \simeq w} - \pi_w \Big|
		> \epsilon  }
\\&\leq \epsilon + \frac{A_\epsilon e^{-C_\epsilon n}}{ \proba\Par{ \rd{W_{[0,2n)}}=\emptystr }}
\end{align*}
According to \cite[Eq.\ (28)]{She11}, \footnote{It has been shown in \cite{GMS15} that $\proba\Par{ \rd{W_{[0,2n)}}=\emptystr }$ decays as a power of $n$, with the exact exponent as a function of $p$. But we do not need this fact here.}
$$
	\lim_{n\to\infty} \frac{1}{n} \log \proba\Par{ \rd{W_{[0,2n)}}=\emptystr } = 0
$$
Therefore the second term converges to zero as $n\to\infty$.
Since $\epsilon$ can be taken arbitrarily close to zero, this shows that $\proba_n\Par{ B\sub{loc}(w,2^{-R}) } \to \pi_w$ as $n\to\infty$.

\subsection{Some properties of the limiting random word}
\label{sub:limit properties}

In this section we show two properties of the infinite random word $W\ind{p}$ which will be the word-counterpart of Theorem \ref{thm:map limit}.
Both properties are true for general $p\in[0,1]$.
However we will only write proofs for $p<1$, since the case $p=1$ corresponds to cFK random maps with parameter $q=\infty$, for which the local limit is explicit.
(The proofs for $p=1$ are actually easier, but they require different arguments.)

\begin{prop}[Sheffield \cite{She11}]\label{prop:pre-continuity}
For all $p\in[0,1]$, almost surely,
\begin{enumerate}
\item	$\rd{W\ind{p}}=\emptystr$, that is, every letter in $W\ind{p}$ is matched.
\item	For all $k\in\integer$, $\rdb{W\ind{p}_{(-\infty,k)}}$ contains infinitely many $\xa$ and infinitely many $\xb$.
\end{enumerate}
\end{prop}

\begin{proof}
The first assertion is proved as Proposition 2.2 in \cite{She11}.
For the second assertion, recall that $\rdb{W\ind{p}_{(-\infty,k)}}$ represents a left-infinite stack of burgers.
Now assume for some $k\in\integer$, it contains only $N$ letters $\xa$ with positive probability.
Then, with probability $\Par{\frac{1-p}{4}}^{N+1}$ and independently of $W\ind{p}_{(-\infty,k)}$, all the $N+1$ letters in $W\ind{p}_{[k,k+N]}$ are $\xA$.
This will leave the $\xA$ at position $k+N$ unmatched in $W$, which happens with zero probability according to the first assertion.
This gives a contradiction when $p<1$.
\end{proof}

For each random word $W\ind{p}$, consider a random walk $Z$ on $\integer^2$ starting from the origin: $Z_0 = (0,0)$, and for all $k\in\integer$,
\begin{equation}\label{eq:random walk}
Z_{k+1} - Z_k = \begin{cases}
	(1,0)	&\text{if }W\ind{p}_k = \xa
\\	(-1,0)	&\text{if }W\ind{p}_k \text{ is matched to an } \xa
\\	(0,1)	&\text{if }W\ind{p}_k = \xb
\\	(0,-1)	&\text{if }W\ind{p}_k \text{ is matched to a } \xb
\end{cases}
\end{equation}
By Proposition \ref{prop:pre-continuity}, $Z_k$ is almost surely well-defined for all $k\in\integer$.
A lot of information about the random word $W$ can be read from $Z$.
The main result of \cite{She11} shows that under diffusive rescaling,
$Z$ converges to a Brownian motion in $\real^2$ with a diffusivity matrix that depends on $p$, demonstrating a phase transition at $p=1/2$.

Let $(X,Y)=Z$. Then $X$ represents the net hamburger count and $Y$ the net cheeseburger count.
Set $i_0 = \sup\set{i<0}{X_i=-1}$ and $j_0 = \inf\set{j>0}{X_j=-1}$.
Let $N_0$ be the number of times that $X$ visits the state $0$ between time $i_0$ and $j_0$.
We shall see in Section \ref{sub:proof} that $N_0$ is exactly the degree of root vertex in the infinite cFK-random map. Below we prove that the distribution of $N_0$ has an exponential tail, that is, there exists constants $A$ and $c>0$ such that $\proba(N_0\geq x)\leq Ae^{-cx}$ for all $x\geq 0$.

\begin{figure}[ht!]
\begin{center}
\includegraphics[scale=0.8]{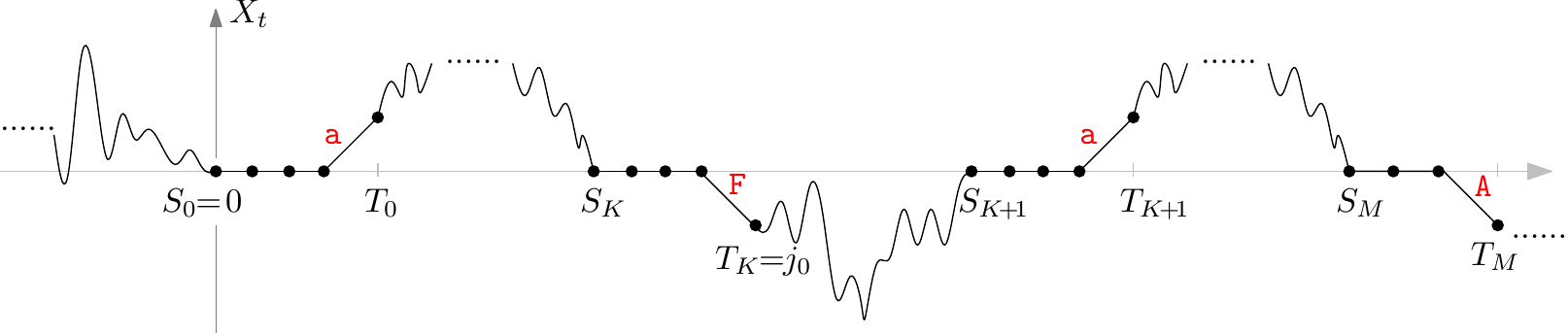}
\caption{The decomposition of $N_0^+$ into intervals $[S_k,T_k]$}.
\label{fig:N0 decomposition}
\end{center}
\end{figure}

\begin{prop}\label{lemma:exponential tail}
$N_0$ has an exponential tail distribution for all $p\in[0,1]$.
\end{prop}

\begin{proof}
First let us consider $N_0^+$,
the number of times that $X$ visits the state $0$ between time $0$ and $j_0$.
Remark that at positive time, the process $X$ is adapted to the natural filtration $(\filtr_k)_{k\geq 0}$, where $\filtr_k$ is the $\sigma$-algebra generated by $W_{[0,k)}$.
Define two sequences of $\filtr$-stopping times $(S_m)_{m\geq 0}$ and $(T_m)_{m\geq 0}$ by $S_0 = 0$, and that for all $m\geq 0$,
\begin{align*}
	T_m		&= \inf\set{k>S_m}{X_k\neq 0}
\\	S_{m+1} &= \inf\set{k>T_m}{X_k	=  0}
\end{align*}
The sequence $S$ (resp. $T$) marks the times that $X$ arrives at (resp. departs from) the state $0$.
Therefore the total number of visits of the state $0$ between time $0$ and $T_m$ is
$\sum_{i=0}^{m} (T_i - S_i)$,
see Fig.~\ref{fig:N0 decomposition}.

By construction, $j_0$ is the smallest $T_m$ such that $X_{T_m} = -1$.
On the other hand, we have $X_{T_m}\!\! \in\{-1,+1\}$ for all $m$ and 
\begin{align*}
	X_{T_m} =+1	\quad	&\iff\quad	W_{T_m} = \xa 			\label{eq:+1 = a}
\\	X_{T_m} =-1 \quad	&\iff\quad	W_{T_m} = \xA\text{ or }
									W_{T_m} \text{ is an } \xF \text{ matched to an } \xa
\end{align*}
Consider the stopping time
\begin{equation}\label{eq:def M}
	M=\inf\set{m\geq 0}{W_{T_m}=\xA}
\end{equation}
Then we have $j_0\leq T_M$, and therefore
\begin{equation}
	N_0^+ \leq \sum_{m=0}^M (T_m - S_m)
\end{equation}
On the other hand, $M$ is the smallest $m$ such that, starting from time $S_m$, an $\xA$ comes before an $\xa$.
Therefore $M$ is a geometric random variable of mean $\frac{1-p}{p}$.

Assume $p<1$ so that $M$ is almost surely finite.
Fix an integer $m\geq 0$.
By the strong Markov property, conditionally to $\{M=m\}$,
the sequence $(T_i-S_i, i=0\ldots m-1)$ is i.i.d.,
and each term in the sequence has the same law as the first arrival time
of $\xa$ in the sequence $(W_k)_{k\geq 0}$ conditioned not to contain $\xA$.
In other words, conditionally to $\{M=m\}$, $(T_i-S_i, i=0\ldots m-1)$ is
an i.i.d.\ sequence of geometric random variables of mean $\frac{1}{2+p}$.
Similarly, conditionally to $\{M=m\}$,
$T_m-S_m$ is a geometric random variable of mean $\frac{1-p}{2+p}$ independent from the sequence $(T_i-S_i, i=0\ldots m-1)$.
Then, a direct computation shows that the exponential moment $\expect[e^{\lambda N_0^+}]$ is finite for some $\gamma>0$.
And by Markov's inequality, the distribution of $N_0^+$ has an exponential tail when $p<1$.

Now we claim that conditionally to the value of $N_0$, the variable 
$N_0^+$ is uniform on $\{1,\ldots,N_0\}$ which implies that $N_0$ also has an exponential tail distribution.
To see why the conditional law is uniform, consider $N_0$ and $N_0^+$ for finite words defined in the same way as for the infinite word $W\ind{p}_\infty$.
Note that for a finite word $w$ the process $X$ does not necessarily hit $-1$ at negative (resp. positive) times.
In this case we just replace $i_0$ (resp. $j_0$) by the infimum (resp. supremum) of the domain of $w$.
Then, $w\mapsto (N_0,N_0^+)$ is a $D\sub{loc}$-continuous function defined on the union of $\cup_{n\geq 0}\W_n$ and the support of $W\ind{p}_\infty$.
Therefore for any integers $k\leq m$,
\begin{equation}
					\proba\ind{p}_n(N_0=m,N_0^+=k)
\cvg{}{n\to\infty}	\proba\ind{p}_\infty(N_0=m,N_0^+=k)
\end{equation}
But, given the sequence of letters in a word, the law $\proba\ind{p}_n$ chooses the letter of index 0 uniformly at random among all the letters.
A simple counting shows that for all $1\!\leq\! k,k'\!\leq\! m$,
we have $\proba\ind{p}_n(N_0=m,N_0^+=k)  =  \proba\ind{p}_n(N_0=m,N_0^+=k')$.
Letting $n\to\infty$ shows that the conditional law of $N_0^+$ given $N_0$ is uniform under $\proba\ind{p}_\infty$.
\end{proof}

\section{The hamburger-cheeseburger bijection}\label{sec:maps}

\subsection{Construction}
In this section we present (a slight variant of) the hamburger-cheeseburger bijection of Sheffield. 
We refer to \cite{She11} for the proof of bijectivity and for historical notes.

We define the hamburger-cheeseburger bijection $\wt{\Psi}$ on a subset of the space $\W$, and it takes values in the space $\wt{\M}$ of \emph{doubly-rooted planar maps with a distinguished subgraph}, that is, planar maps with two distinguished corners and one distinguished subgraph.
We can write this space as
\begin{equation}
	\wt{\M} = \set{(\map{M,G,s})}{ (\map{M,G})\in\M
								\text{ and }\map{s}\text{ is a corner of }\map{M}}
\end{equation}
Note that the second root $\map{s}$ may be equal to or different from the root of $\map{M}$.
We define in the same way $\wt{\M}_n$, the doubly-rooted version of the space $\M_n$.
Its cardinal is $2n$ times that of $\M_n$.

We start by constructing $\wt{\Psi}:\W_n\to\wt{\M}_n$ in three steps.
The first step transforms a word in $\W_n$ into a decorated planar map called \emph{arch graph}.
The second and the third step apply graph duality and local transformations to the arch graph to get a \emph{tree-rooted map}, and then \emph{a subgraph-rooted map} in $\wt{\M}_n$.

\paragraph{Step 1: from words to arch graphs.}
Fix a word $w\in\W_n$.
Recall from Section \ref{sec:words} the construction of the non-crossing arch diagram associated to $w$.
In particular since $\overline{w}= \emptyset$, there is no half-arch.
We link neighboring vertices by unit segments $[j-1,j]$ and link the first vertex to the last vertex by an edge that wires around the whole picture on either side of the real axis, without intersecting any other edges.
This defines a planar map $\map{A}$ of $2n$ vertices and $2n$ edges.
In $\map{A}$ we distinguish edges coming from arches and the other edges.
The latter forms a simple loop passing through all the vertices.

We further decorate $\map{A}$ with additional pieces of information.
Recall that the word $w$ is indexed by an interval of the form $I_k=\{-k,\ldots,2n-1-k\}$ where $0\leq k<2n$. We will mark the oriented edge $\map{r}$ from the vertex $0$ to the vertex $-1$, and the oriented edge $\map{s}$ from the first vertex ($-k$) to the last vertex ($2n-1-k$).
If $k=0$, then $\map{r}$ and $\map{s}$ coincide.
Furthermore, we mark each arch incident to an $\xF$-vertex by a star $*$. (See Fig.~\ref{fig:step 1-2})
We call the decorated planar map $A$ the \emph{arch graph} of $w$.
One can check that it completely determines the underlying word $w$.

\paragraph{Step 2: from arch graphs to tree-rooted maps.}
We now consider the dual map $\map{\Delta}$ of the arch graph $\map{A}$.
Let $\map{Q}$ be the subgraph of $\map{\Delta}$ consisting of edges whose dual edge is on the loop in $\map{A}$.
We denote by $\map{\Delta\backslash Q}$ the set of remaining edges of $\map{\Delta}$ (that is, the edges intersecting one of the arches).

\begin{prop} The map $\map{\Delta}$ is a triangulation, the map $\map{Q}$ is a quadrangulation with $n$ faces and $\map{\Delta}\backslash \map{Q}$ consists of two trees.
\qed
\end{prop}

We denote by $\map{T}$ and $\map{T}^\dual$ the two trees in $\map{\Delta\backslash Q}$, with $\map{T}$ corresponding to faces of the arch graph in the upper half plane.
Then $\map{Q}$,\! $\map{T}$\! and\! $\map{T}^\dual$ form a partition of edges in the triangulation $\map{\Delta}$.
Note that $\map{T}$ and $\map{T}^\dual$ give the (unique) bipartition of vertices of $\map{Q}$.
Let $\m $ be the planar map associated to $\map{Q}$ by Tutte's bijection, such that $\m $ has the same vertex set as $\map{T}$.
(The latter prescription allows us to bypass the root, and define Tutte's bijection from unrooted quadrangulations to unrooted maps.)
We thus obtain a couple $(\map{M,T})$ in which $\m $ is a map with $n$ edges and $\map{T}$ is spanning tree of $\m $.
Remark that $\map{T}^\dual$ is the dual spanning tree of $\map{T}$ in the dual map $\m^\dual$.
This relates the duality of maps with the duality on words which consists of exchanging $\xa$ with $\xb$ and $\xA$ with $\xB$.

Fig. \refnote{fig:step 1-2}{(a)} summarizes the mapping from words to tree-rooted maps (Step 1 and 2) with an example.
Note that we have omitted the two roots and the stars on the arch graph in the above discussion.
But since graph duality and Tutte's bijection provide canonical bijections between edges, the roots and stars can be simply transferred from the arches in $\map{A}$ to the edges in $\m $.
With the roots and stars taken into account, it is clear that $w\mapsto(\map{M,T})$ is a bijection from $\W_n$ onto its image.

\begin{figure}[ht!]
\captionsetup{width=0.8\textwidth}
\begin{center}
\includegraphics[scale=1]{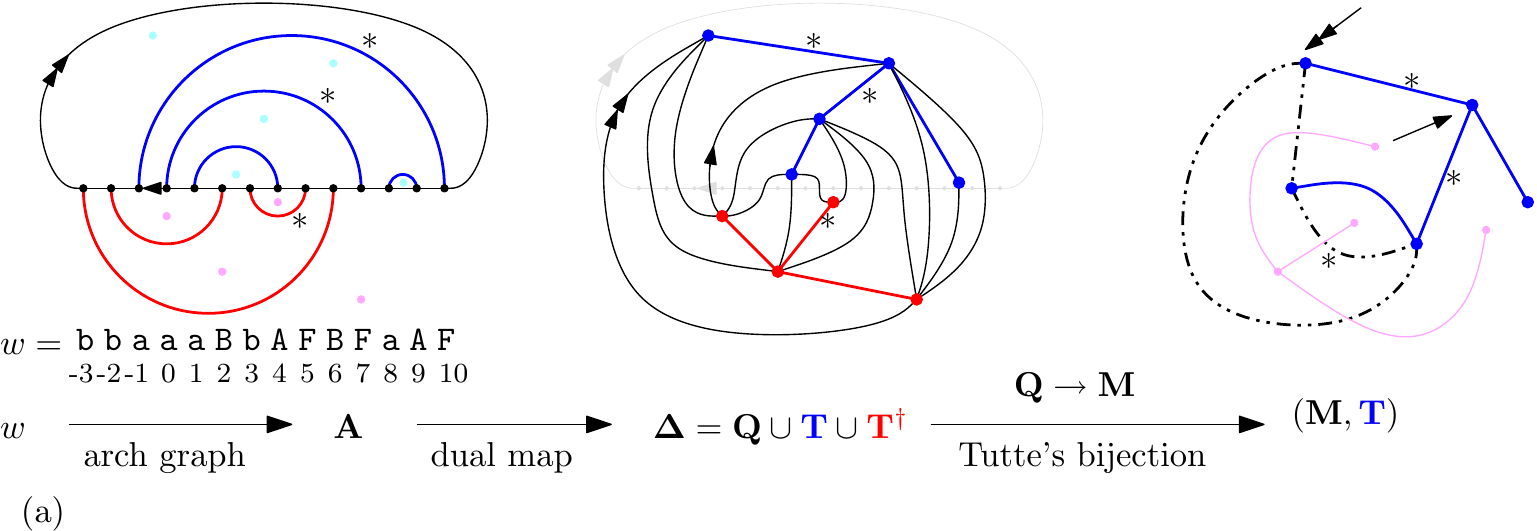}
\includegraphics[scale=1]{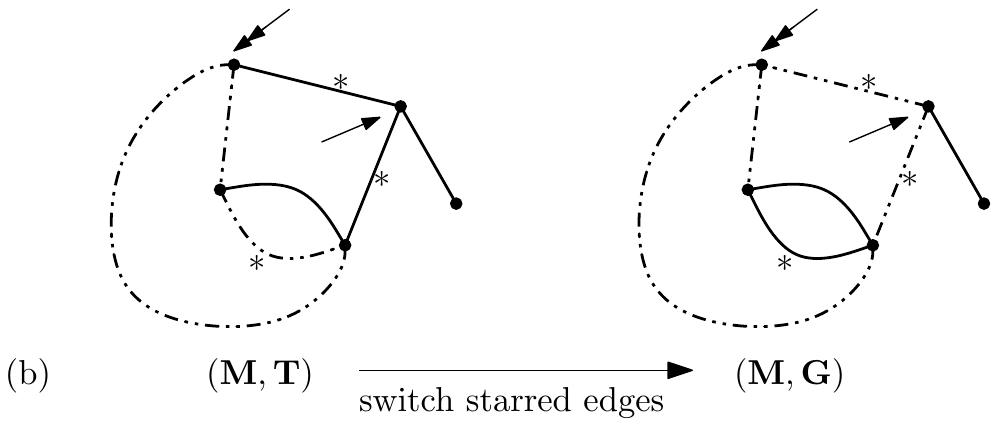}
\caption{Construction of the hamburger-cheeseburger bijection.
(a) From word to tree-rooted map. (b) From tree-rooted map to subgraph-rooted map.
}
\label{fig:step 1-2}
\end{center}
\end{figure}
\paragraph{Step 3: from tree-rooted maps to subgraph-rooted maps.}
Now we ``switch the status'' of every starred edge in $\m $ relative to the spanning tree $\map{T}$.
That is, if a starred edge is not in $\map{T}$, we add it to $\map{T}$; if it is already in $\map{T}$, we remove it from $\map{T}$.
Let $\map{G}$ be the resulting subgraph.
See Fig. \refnote{fig:step 1-2}{(b)} for an example.

Recall that there are two marked corners $\map{r}$ and $\map{s}$ in the map $\m $.
By an abuse of notation, from now on we denote by $\m $ the \emph{rooted} map with root corner $\map{r}$.
Then, the hamburger-cheeseburger bijection is defined by $\wt{\Psi}(w) = (\map{M,G,s})$.
Let $\Psi(w) = (\map{M,G})$ be its projection obtained by forgetting the second root corner.
We denote by $\card{\xF}w$ the number of letters $\xF$ in $w$, and by $\ell$ the number of loops associated to the corresponding subgraph-rooted map $(\map{M,G})$.
\begin{thm}[Sheffield \cite{She11}] \label{thm:shebij}
The mapping $\wt{\Psi}:\!\W_n\!\to\!\wt{\M}_n$ is a bijection such that $\ell=1+\card{\xF}w$ for all $w\in\W_n$.
And $\,\proba[Q]\ind{q}_n$ is the image measure of $\,\proba\ind{p}_n$ by $\Psi$ whenever
$$p=\frac{\sqrt{q}}{2+\sqrt{q}}.$$
\end{thm}
\proof The proof of this can be found in \cite{She11}.
However we include a proof of the second fact to enlighten the relation $ p = \frac{ \sqrt{p}}{2 + \sqrt{q}}$.
For $w\in\W_n$, since $\rd{w}=\emptystr$, we have
			$\card\xa w + \card\xb w = \card\xA w + \card\xB w + \card\xF w = n$.
Therefore, when $p=\frac{\sqrt{q}}{2+\sqrt{q}}$,
		\begin{align*}
			\proba\ind{p}_n(w)
				&\propto\Par{\frac{1}{4}}^{\card\xa w + \card\xb w}
						 \Par{\frac{1-p}{4}}^{\card\xA w + \card\xB w}
						 \Par{\frac{p}{2}}^{\card\xF w}
			\\	&	=	\Par{\frac{1}{4}}^n
						\Par{\frac{1-p}{4}}^{n-\card\xF w}
						\Par{\frac{p}{2}}^{\card\xF w}
				\propto\Par{\frac{2p}{1-p}}^{\card\xF w}	= \sqrt{q}^{\ell-1}
		\end{align*}
After normalization, this shows that $\proba[Q]\ind{q}_n$ is the image measure of $\proba\ind{p}_n$ by $\Psi$.
\endproof 

\begin{prop} \label{prop:extcontinu}
We can extend the mapping $\Psi$ to $\W\to\M$ so that it is $\proba\ind{p}_\infty$-almost surely continuous with respect to $D\sub{loc}$ and $d\sub{loc}$, for all $p\in[0,1]$.
\end{prop}

\begin{figure}[t!]
\captionsetup{width=0.8\textwidth}
\begin{center}
\includegraphics[scale=0.7]{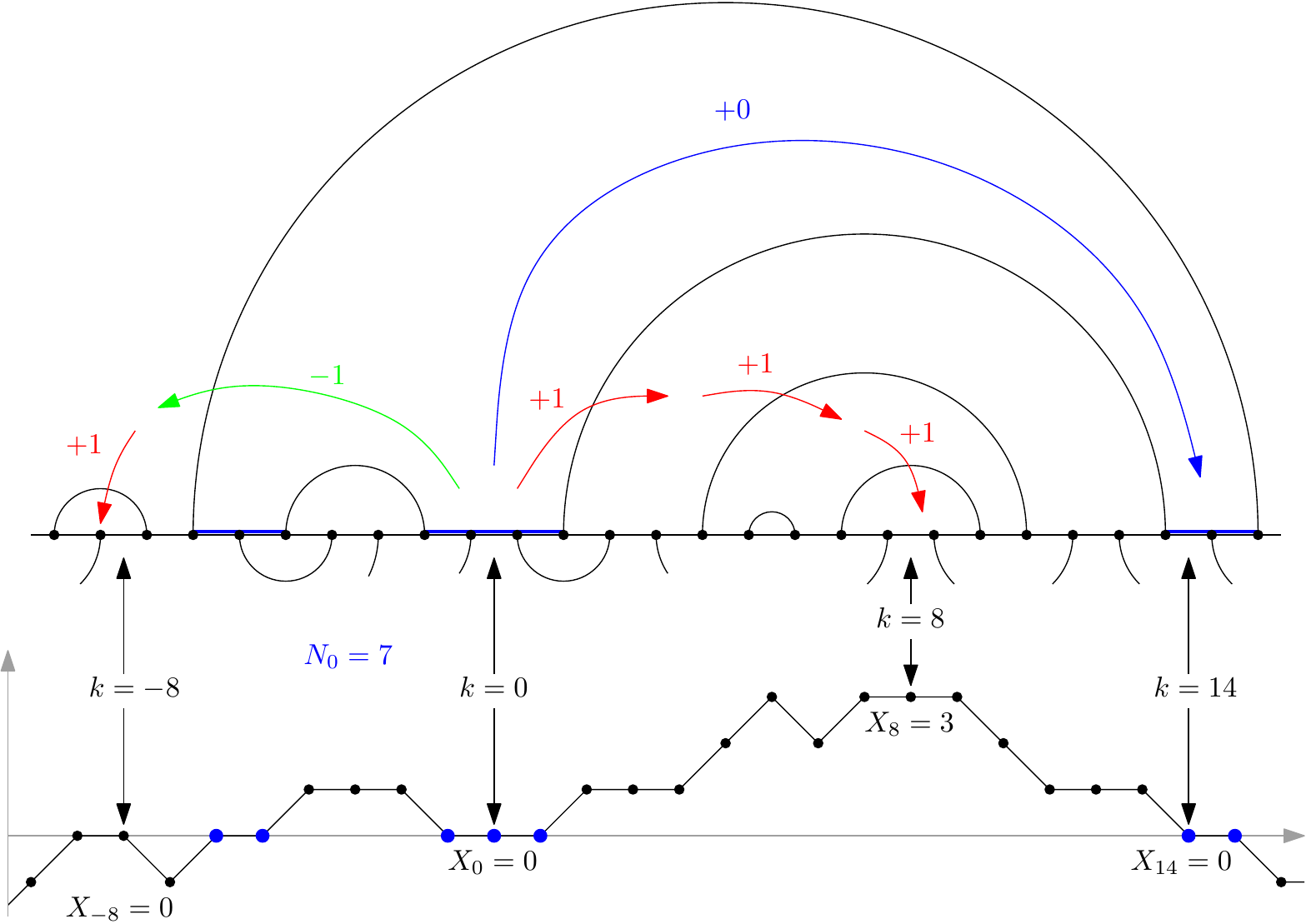}
\caption{An example of an infinite arch graph, and the associated process $(X_k)_{k\in\integer}$.
We shifted the arch graph horizontally by $1/2$ relative to the graph of $X$, since the time $k$ for the process $X$ naturally corresponds to the interval $[k-1,k]$ in the arch graph.}
\label{fig:interpretation of X}
\end{center}
\end{figure}

\proof Observe that if we do not care about the location of the second root $\map{s}$, then the word $w$ used in the construction of $\Psi$ does not have to be finite.
Set
\begin{equation}\label{eq:W infty}
	\W_\infty = \set{w\in\Theta^\integer}{
	\begin{array}{r}
	\rd{w}=\emptystr\text{ and for all } k\in\integer,\rdb{w_{(-\infty,k)}}\text{ contains}
\\	\text{infinitely many }\xa\text{ and infinitely many }\xb 
	\end{array}
	}
\end{equation}
We claim that indeed, for each $w\in \W_\infty$, Step 1, 2 and 3 of the construction define a (locally finite) infinite subgraph-rooted map: as in the case of finite words, the condition $\rd{w}=\emptystr$ ensures that the arch graph $\map{A}$ of $w$ is a well-defined infinite planar map (that is, all the arches are closed).
To see that its dual map $\map{\Delta}$ is a locally finite, infinite triangulation, we only need to check that each face of $\map{A}$ has finite degree.
Observe that a letter $\xa$ in $w$ appears in $\rdb{w_{(-\infty,k)}}$ if and only if it is on the left of $w_k$, and that its partner is on the right of $w_k$.
This corresponds to an arch passing above the vertex $k$.
Therefore, the remaining condition in the definition of $\W_\infty$ says that there are infinitely many arches which pass above and below each vertex of $\map{A}$.
This guarantees that $\map{A}$ has no unbounded face.
The rest of the construction consists of local operations only.
So the resulting subgraph-rooted map $(\map{M,G}) = \Psi(w)$ is a locally finite subgraph-rooted map.

Also, by Proposition \ref{prop:pre-continuity}, we have $ \mathbb{P}_{\infty}^{(p)}( \mathcal{W}_{\infty})=1$.  It remains to see that (the extension) of $\Psi$ is continuous on $  \mathcal{W}'=\bigcup_{n} \mathcal{W}_{n} \cup \mathcal{W}_{\infty}$. Let $w^{(n)}, w \in \mathcal{W}'$ so that $w^{(n)} \to w$ for $D\sub{loc}$. If $w$ is finite, there is nothing to prove. Otherwise, let $(\map{M,G})=\Psi(w)$ and consider a ball $B$ of finite radius $r$ around the root in the map $\m $. By locality of the mapping $\map{\Delta}\leadsto\m $, the ball $B$ can be determined by a ball $B'$ of finite radius $r'$ (which may depend on $ \mathbf{M}$) in $\map{\Delta}$.
But each triangle in $\map{\Delta}$ corresponds to a letter in the word, so
there exists $r''$ (which may depend on $\mathbf{M}$) such that if $w^{(n)}_{[-r'',r'']}= w_{[-r'',r'']}$ then the balls of radius $r'$ in $\map{\Delta}$ coincide.
This proves that $\Psi$ is continuous on $\W_\infty$.
\endproof

\subsection{Proof of Theorem \ref{thm:map limit}}\label{sub:proof}
Combining Proposition \ref{prop:word limit}, Theorem \ref{thm:shebij} and Proposition \ref{prop:extcontinu} yields the convergence statement of the theorem. The self-duality of the infinite cFK random map follows from the finite self-duality. It remains to show that the infinite cFK random maps are almost surely one-ended and recurrent for all $q$, and that their laws are mutually singular for different values of $q$.

\paragraph{One-endedness.} Recall that a graph $\map{G=(V,E)}$ is said to be one-ended if for any finite subset of vertices $\map{U}$, $\map{V\backslash U}$ has exactly one infinite connected component.
We will prove that for a word $w\in\W_\infty$, $\Psi(w)$ is one-ended.
Let $\map{A}$ (resp. $\map{\Delta}$) be the arch graph (resp. triangulation) associated to $w$, and let $(\map{M,G})=\Psi(w)$.
By the second condition in the definition of $\W_\infty$ (see \eqref{eq:W infty}), there exist arches that connect vertices on the left of $w_{-R}$ to vertices on the right of $w_R$ for any finite number $R$.
Therefore the arch graph $\map{A}$ is one-ended. It is then an easy exercise to deduce from this that the triangulation $\map{\Delta}$ and then the map $ \mathbf{M}$ are also one-ended.

\paragraph{Recurrence.} To prove the recurrence of $\m $ we use the general criterion established by  Gurel-Gurevich and Nachmias \cite{GGN13a}. Notice first that under $ \mathbb{Q}_{n}^{(q)}$, the random maps are  \emph{uniformly rooted}, that is, conditionally on the map, the root vertex $\rho$ is chosen with probability proportional to its degree. By \cite{GGN13a} it thus suffices to check that the distribution of $\deg(\rho)$ has an exponential tail. For this we claim that the variable $N_{0}$ studied in Lemma \ref{lemma:exponential tail} exactly corresponds to the degree of the root in an infinite cFK random map.
From the construction of the hamburger-cheeseburger bijection, we see that the vertices of the map $\m $ corresponds to the faces of the arch graph \emph{in the upper half plane}.
In particular, the root vertex $\rho$ corresponds to the face \emph{above} the interval $[-1,0]$, and $\deg(\rho)$ is the number of unit intervals on the real axis which are also on the boundary of this face.
On the other hand, $X_k$ is the \emph{net number} of arches that one \emph{enters} to get from the face above $[-1,0]$ to the face above $[k-1,k]$, see Fig.~\ref{fig:interpretation of X}.
So $N_{0}$ exactly counts the above number of intervals.

\paragraph{Mutual singularity.}
The basic idea is to construct a measurable function $f$ on infinite rooted planar maps such that $f(\m )$ is almost surely constant for each $q\in[0,\infty]$, and that $f(\m ) = f(q)$ is an injective function of $q$.
One such function can be defined as follows.
Consider the simple random walk on the map $\m $.
We regard it as a sequence of oriented edges $(\ve[E]_k,k\ge 0)$ such that $\ve[E]_0$ is the the root edge (i.e.\ the edge on the left of the root corner), and such that $\ve[E]_{k+1}$ starts at the vertex where $\ve[E]_k$ ends.
An oriented edge $\ve$ is \emph{pending} if the starting point of $\ve$ has degree 1.
Let
$$ f(\m ) = \lim_{n\to\infty} \frac{1}{n}\sum_{k=0}^{n-1} \idd{\ve[E]_k \text{ is pending}} $$
From the ergodicity result in the appendix and Birkhoff's ergodic theorem, it follows that $f(\m )=\proba[Q](\ve[E]_0 \text{ is pending})$ almost surely.
Recall that $\ve[E]_0$ is the root edge of $\m $.
A moment of look at Fig. \ref{fig:step 1-2} shows that $\ve[E]_0$ is a pending edge in $(\m,\map{G})=\Phi(w)$ if and only if $w_{-1}=\xa$ and $w_0\in\{\xA,\xF\}$.
Therefore,
\begin{align*}
	\proba[Q](\ve[E]_0 \text{ is pending})
=&	\proba(w_{-1}=\xa\text{ and }w_0\in\{\xA,\xF\})
\\=&\theta(\xa)\bpar{ \theta(\xA)+\theta(\xF) } = \frac{1+p}{16}
=	\frac{1}{16} \Bpar{ 1+\frac{\sqrt{q}}{2+\sqrt{q}} }
\end{align*}
We see that $f(q)=f(\m )$ is injective in $q$.

\appendix
\section{Ergodicity of cFK random maps}\label{appendix}
Here we use the framework set up by Benjamini and Curien in \cite{BC12} and we follow the exposition in \cite[Sec 3.1]{PSHIT}.
Consider the space $\overrightarrow{\M}$ of all locally finite rooted maps endowed with a path $\ve = \ven$ starting from the root edge.
Let $\vec{d}\sub{loc}$ be the natural local distance on $\overrightarrow{\M}$
$$ \vec{d}\sub{loc}((\m,\ve),(\m',\ve\,')) = 
	\inf\set{2^{-R}}{B_R(\m)=B_R(\m')
					\text{ and } \ve_k=\ve\,'_k\text{ for }0\le k\le R}
$$
and consider the associated Borel $\sigma$-algebra.
If $\m $ is a rooted map, let $\QM$ be the law of the simple random walk starting from the root edge of $\m $.
We denote by $\QQ$ the probability measure
$$
	\QQ(A) = \int \QM(A_\m)\, \proba[Q](\dd \m)
$$
on $\overrightarrow{\M}$, where $A_\m  = \set{ \ven }{ (\m,\ven)\in A}$.
To simplify notation, we will write $\QM(A)$ instead of $\QM(A_\m )$ in the sequel.

If $\ve$ is an oriented edge of $\m $, we write $\m|_{\ve}$ for the map obtained by re-rooting $\m $ at $\ve$.
Observe that if $(\m,(\ve[E]_n)_{n\ge 0})$ is a random variable of law $\QQ$, then 
$$	\m  = \m|_{\ve[E]_1}	$$
in distribution. (Remark that the left-hand-side is the same as $\m|_{\ve[E]_0}$.)
This is a consequence of the fact that $\m $ is the local limit of a sequence of uniformly rooted finite maps.
It follows that the shift operator
$$	\tau: (\m,\ven) \mapsto (\m|_{\ve_1},\ven[+1])	$$
preserves the measure $\QQ$.

\begin{prop}
The shift operator $\tau$ is ergodic for $\QQ$.
\end{prop}

\begin{proof}
Let $A$ be a $\tau$-invariant event for $\QQ$, that is, $\QQ(A\Delta \tau^{-1}(A))=0$, where $\Delta$ denotes the symmetric difference: $A\Delta B = (A\backslash B)\cup (B\backslash A)$.
We do the proof in three steps:

\begin{enumerate}[label=(\alph*)]
\item	$\proba[Q]$-almost surely, $\QM(A)\in \{0,1\}$.
\end{enumerate}
Recall that the simple random walk on $\m $ is $\proba[Q]$-almost surely recurrent.
Therefore, it can be decomposed into an i.i.d.\ sequence of excursions from the root edge.
An event invariant by $\tau$ is also invariant by the shift operator associated to this i.i.d.\ sequence.
Thus (a) follows from Kolmogorov's zero-one law.

\begin{enumerate}[resume,label=(\alph*)]
\item	$\proba[Q]$-almost surely, the value of $\QM(A)$ is invariant under any re-rooting of the map $\m$.
\end{enumerate}
First let us fix a rooted map $\map{M}$.
Let $x$ be the vertex to which the root edge points.
Denote by $d$ the degree of $x$, and $\vec{\rho}_1,\ldots,\vec{\rho}_d$ the $d$ oriented edges that point away from $x$.
We deduce from the Markov property of the simple random walk that
\newcommand*{\bset}[2]{\big\{\,#1\ \big|\ #2\,\big\}}
\begin{align*}
\QM(\tau^{-1}A) &= \QM\bset{ \ven }{ (\m|_{\ve_1},\ven[+1])\in A }
\\				&= \sum_{k=1}^d \QM\bset{ \ven }{ (\m|_{\ve_1},\ven[+1])\in A 
									\text{ and } \ve_1 = \vec{\rho}_k }
\\				&= \sum_{k=1}^d\, \frac{1}{d} Q_{\m|_{\vec{\rho}_k}}\!\!
											\bset{ \ven }{ (\m|_{\vec{\rho}_k},\ven)\in A }
\ =\ \frac{1}{d} \sum_{k=1}^d Q_{\m|_{\vec{\rho}_k}}\!(A)
\end{align*}
From
$$
	\int \abs{\QM(A) - \QM(\tau^{-1}A)} \, \proba[Q](\dd \m)
\le \QQ(A\Delta \tau^{-1}A) =0
$$
we deduce that $\proba[Q]$-almost surely,
$$
	\QM(A) = \frac{1}{d} \sum_{k=1}^d Q_{\m|_{\vec{\rho}_k}}\!(A)
$$
But according to (a), $\QM(A)$ and $Q_{\m|_{\vec{\rho}_k}}\!(A)$ ($k=1,\ldots,d$) are either 0 or 1.
Thus $\proba[Q]$-almost surely, $\QM(A) = 1$ if and only if for all $k=1,\ldots,d$, $Q_{\m|_{\vec{\rho}_k}}\!(A)=1$.
In other words, the value of $\QM(A)$ is unchanged when re-rooting at a neighbor of the root edge.
But since the maps are connected, we obtain (b) by iterating the above argument.

\begin{enumerate}[resume,label=(\alph*)]
\item 	If an event $\hat{A}$ on the space of (locally finite) rooted maps is $\proba[Q]$-almost surely invariant under re-rooting, then $\proba[Q](\hat{A})\in\{0,1\}$.
\end{enumerate}
Consider the measure-preserving mapping $\hat{\Phi}:w\mapsto \m$, where $(\m,\map{G})=\Phi(w)$ for some subgraph $\map{G}$.
Via this mapping, a translation of indices in a word $w$ give rise to a re-rooting of the corresponding map $\m$.
Therefore if an event $\hat{A}$ is $\proba[Q]$-almost surely invariant under re-rooting, then $\hat{\Phi}^{-1}\!\hat{A}$ is $\proba$-almost surely invariant under translation of the indices.
But, under $\proba$ the letters of $w$ are i.i.d.\ random variables, so we have $\proba[Q](\hat{A}) = \proba(\hat{\Phi}^{-1}\hat{A})\in\{0,1\}$.

\medskip\noindent
Finally, considering $\hat{A}=\{\m : \QM(A)=1\}$ shows that $\QQ(A)\in\{0,1\}$, as desired.
\end{proof}

\paragraph{Remark.} Only the step (c) of the proof uses specific features of the cFK random maps, namely, their representation by an i.i.d.\ sequence of letters.
For any infinite random map whose law is stationary under $\tau$, the proof of (a) and (b) goes through provided that the random map is almost surely recurrent. (Note that the proof of (b) depends on (a).)
The case of almost surely transient random maps was treated in \cite[Prop.\ 10]{PSHIT}.
There, (a) was proved under the following reversibility condition: if $\overleftarrow{e}$ is the root edge of $\map{M}$, oriented in the opposite direction, then
$$
	\m = \m|_{\overleftarrow{e}}
$$
in distribution.
And (b) was replaced by the following variant, which follows directly from the transience of the map.
\begin{enumerate}[label=(\alph*')]
\setcounter{enumi}{1}
\item	Almost surely, $\QM(A)$ is unchanged by any finite modification of the map $\m$.
\end{enumerate}

\bibliographystyle{abbrv}{%
\bibliography{cite_draft}}

\bigskip
\noindent
Linxiao Chen, Institut de Physique Théorique, Université Paris Saclay, CEA, CNRS, F-91191 Gif-sur-Yvette, France.\\
Laboratoire de Mathématiques d'Orsay, Univ. Paris-Sud, CNRS, Université Paris-Saclay, 91405 Orsay, France.

\end{document}